\documentclass[12pt, a4paper]{article}
\usepackage{graphicx}
\usepackage{amssymb}
\usepackage{stmaryrd}
\usepackage{enumitem}

\usepackage{geometry}
\usepackage{tikz}
\usepackage{float}
\usepackage{url}
\usepackage{mathtools}
\usepackage{amsmath,amsfonts,amsthm}
\usepackage{mathrsfs}

\newtheorem{theorem}{Theorem}[section]
\newtheorem{lemma}[theorem]{Lemma}
\newtheorem{proposition}{Proposition}[section]

\newtheorem{definition}{Definition}

\newtheorem{remark}{Remark}

\usepackage{pgfplots}
\pgfplotsset{compat=1.18}

\newcommand{\dens}{\mathrm{dens}}

\usepackage{subfig}

\usepackage{ytableau}
\usepackage[colorlinks,
linkcolor=blue,
anchorcolor=blue,
citecolor=blue
]{hyperref}
\begin{document}
\begin{center}
{\large \bf  Analysis of the Density of Words under Morphism $\{a,b\}$}
\end{center}
\begin{center}
 Jasem Hamoud$^1$ \hspace{0.2 cm}  Duaa Abdullah$^{2}$\\[6pt]
 $^{1,2}$ Physics and Technology School of Applied Mathematics and Informatics \\
Moscow Institute of Physics and Technology, 141701, Moscow region, Russia\\[6pt]
Email: 	 $^{1}${\tt khamud@phystech.edu},
 $^{2}${\tt abdulla.d@phystech.edu}
\end{center}
\noindent
\begin{abstract}
In this paper, we analyze the density of the  Fibonacci word and its derived forms by examining the morphisms associated with each. It offers a comparative analysis of the density of Fibonacci numbers alongside other words derived from Fibonacci word. Fibonacci words over the alphabet $\{a,b\}$, we define a novel \emph{power} operation that yields a formal linear combination in the free abelian group generated by all finite words.
\end{abstract}

\noindent\rule{12.2cm}{1.0pt}

\noindent\textbf{AMS Classification 2010:}  05C42, 11B05, 11R45, 11B39, 68R15.

\noindent\textbf{Keywords:} Morphism, Density, Fibonacci word, Power.

\noindent\rule{12.2cm}{1.0pt}

\section{Introduction}
Throughout this paper, let $\{F_k\}_{k \ge 1} = \{F_1, F_2, F_3, \ldots\}$ be a sequence satisfying the recurrence relation
\[
F_k = F_{k-1} + F_{k-2} \quad \text{for all } k \ge 3,
\]
which is called a \emph{Fibonacci sequence}. Such a sequence is uniquely determined by its first two terms $F_1$ and $F_2$.

Similarly, in~\cite{Rebman1975} the \emph{Lucas numbers} $\{L_n\}_{n \ge 1}$ are introduced as a Fibonacci sequence with initial terms $L_1 = 1$ and $L_2 = 3$. The \emph{Fibonacci word} can also be defined in terms of Beatty's theorem~\cite{Beatty1926} using the golden ratio. Let $\varphi = (1+\sqrt{5})/2$ and $\varphi^2 = \varphi + 1$. Consider the infinite word $(u_k)_{k \ge 1}$ over the alphabet $\{x,y\}$ defined by
\[
u_k =
\begin{cases}
x, & \text{if } k \in \{\lfloor n \varphi \rfloor : n \ge 1\},\\[2mm]
y, & \text{if } k \in \{\lfloor n \varphi^2 \rfloor : n \ge 1\}.
\end{cases}
\]
By Beatty's theorem, the sets $\{\lfloor n \varphi \rfloor : n \ge 1\}$ and $\{\lfloor n \varphi^2 \rfloor : n \ge 1\}$ form a partition of $\mathbb{N}_{\ge 1}$, and the resulting word $(u_k)_{k \ge 1}$ is a Sturmian word; in particular, it is not ultimately periodic.

Each positive integer $m$ has a unique representation, called the \emph{Zeckendorf representation}~\cite{Shah2021}, of the form
\[
m = \sum_{i=1}^{\infty} r_i F_{i+1},
\]
where $r_i \in \{0,1\}$ for all $i$ and $r_i r_{i+1} = 0$ for all $i \ge 1$; that is, no two consecutive Fibonacci numbers occur in the sum. Let $G$ denote the infinite Fibonacci word over a binary alphabet. If $w$ is a factor of $G$~\cite{Chuan2005}, the \emph{location set} of $w$ is defined by
\[
\lambda(w) = \{ m \in \mathbb{Z}^+ : G[m;|w|] = w \},
\]
where $G[m;|w|]$ denotes the factor of $G$ of length $|w|$ starting at position $m$. Given a positive integer $t$, choose $k$ such that $F_k \le t < F_{k+1}$.

Following Anisimov and Zavadskyi~\cite{Anisimov2017}, let $\{0,1\}^*$ denote the set of all finite \emph{binary strings} over the alphabet $\{0,1\}$. For a non-negative integer $m$, the notation $1^m$ (respectively, $0^m$) denotes the word consisting of $m$ consecutive ones (respectively, $m$ consecutive zeros), and the \emph{empty word}, denoted by $\varepsilon$, corresponds to the case $m=0$. An \emph{isolated run} of consecutive ones in a word $w$ is defined as a maximal block of ones that satisfies one of the following: it is a prefix of $w$ ending in zero; it is a suffix of $w$ starting with zero; it is a factor of $w$ surrounded on both sides by zeros; or it coincides with the whole word $w$. The length of a word $w$ over $\{0,1\}$ is denoted by $|w|$.

The existence of infinite square-free words over a three-letter alphabet~\cite{Berstel2007} provides a classical example of pattern avoidance. Formally, a pattern $p$ is said to \emph{occur} in a word $w$ if there exists a non-erasing substitution $h$ such that $h(p)$ appears as a factor of $w$. A substitution $h$ is \emph{non-erasing} (or \emph{length-increasing}) if it maps no symbol to the empty word. If $p$ does not occur in $w$, the word $w$ is said to \emph{avoid} the pattern $p$. Equivalently, the \emph{pattern language} of $p$ over an alphabet $A$ is the set of all words $h(p)$ with $h$ ranging over non-erasing substitutions $h : \operatorname{Alph}(p)^* \to A^*$. Thus $w$ avoids $p$ precisely when none of its factors belongs to this pattern language.

Particular infinite words of interest can be constructed by iterating morphisms (homomorphisms of free monoids), also called \emph{substitutions}~\cite{Berthé2016}. A map $h : \Sigma^* \to \Delta^*$, where $\Sigma$ and $\Delta$ are alphabets, is a \emph{morphism} if
\[
h(xy) = h(x)h(y) \quad \text{for all } x,y \in \Sigma^*.
\]
It suffices to specify $h(a)$ for each letter $a \in \Sigma$ in order to define $h$ on all of $\Sigma^*$. For example, consider the morphism $h : \{0,1,2\}^* \to \{0,1,2\}^*$ given by
\[
0 \mapsto 01201,\quad 1 \mapsto 020121,\quad 2 \mapsto 0212021.
\]
A morphism $h : \Sigma^* \to \Sigma^*$ is said to be \emph{prolongable} on a letter $a \in \Sigma$ if there exists a word $x \in \Sigma^*$ such that $h(a) = ax$ and $h^i(x) \ne \varepsilon$ for all integers $i \ge 0$. In this case, the limit
\[
h^\omega(a) = \lim_{n \to \infty} h^n(a)
\]
exists in the prefix topology and defines an infinite word that is a fixed point of $h$.

The aim of this work is to analyze the density of the Fibonacci word and a family of words derived from it by studying the morphisms associated with these words. In addition, a comparative analysis of the density of Fibonacci numbers and the density of the corresponding derived words is carried out.

The main contributions of this paper can be summarized as follows. 
First, we develop a unified morphic and mechanical description of the infinite Fibonacci word and use it to obtain explicit asymptotic densities of symbols, together with sharp $O(1)$ discrepancy bounds for finite prefixes, thereby giving a self-contained account of density phenomena tightly linked to complementary Beatty sequences and Sturmian dynamics. 
Second, we extend the density analysis from the classical Fibonacci word to a family of morphic words over the alphabet $\{a,b\}$, including framed Fibonacci-like words, and we show that the limiting frequencies of letters in these derived words are again governed by the golden ratio, which clarifies how letter densities behave under natural morphic constructions. 
Finally, we introduce a novel ``power'' operation in the free abelian group generated by all finite words over $\{a,b\}$ and prove that a certain noncommutative polynomial built from finite Fibonacci words is in fact independent of the index, yielding a new invariant in combinatorics on words and linking density-type arguments with identities in the noncommutative algebra $\mathbb{Z}\langle a,b\rangle$.
\section{Preliminaries}\label{sec2}
In this section, we recall some fundamental concepts from symbolic dynamics. Let $T$ be a map acting on a suitable space $X$ and consider the orbit $(T^n(x))_{n \ge 0}$ of a point $x \in X$. Infinite words arise naturally as symbolic codings of such orbits, typically with some loss of information, and their combinatorial properties reflect dynamical features of the underlying system. The $\beta$-transformation, related to $\beta$-expansions of real numbers, admits a natural dynamical interpretation and will be discussed separately.

To illustrate a concrete example, recall the base-$b$ expansion of real numbers in $[0,1)$. The greedy algorithm can be interpreted dynamically via the iteration of the map
\[
T_b : [0,1) \to [0,1), \quad y \mapsto \{by\},
\]
where $\{\cdot\}$ denotes the fractional part. This map partitions $[0,1)$ into $b$ subintervals $[j/b,(j+1)/b)$ for $j = 0,\dots,b-1$. For each $i \ge 0$, if $T_b^i(x)$ lies in $[j/b,(j+1)/b)$, then the $i$-th digit $c_i$ in the base-$b$ expansion $\operatorname{rep}_b(x)$ equals $j$. Basic combinatorial facts, such as the total number of words of a given length over a fixed alphabet, were known long before the systematic study of word combinatorics. The notion of a \emph{partition word} is recalled in Definition~\ref{defpren1}, and Definition~\ref{defpren2} reviews the standard concept of a subsequence.

\begin{definition}[\cite{Heubach2009}]\label{defpren1}
Let $w = w_1 \cdots w_n$ be a word over an ordered alphabet of integers. The word $w$ is called a \emph{partition word} if, for any integers $i<j$, the first occurrence of $i$ in $w$ precedes the first occurrence of $j$. Equivalently, each letter satisfies $w_k \le k$ for all $1 \le k \le n$.
\end{definition}

\begin{definition}[\cite{Elzinga2008}]\label{defpren2}
Let $\{a_n\}_{n \in I}$ be a sequence indexed by a set $I \subseteq \mathbb{N}$. A sequence $\{b_n\}_{n \in I'}$ is called a \emph{subsequence} of $\{a_n\}_{n \in I}$ if there exists a strictly increasing sequence of indices $\{i_n\}_{n \in I'}$ with $i_n \in I$ for all $n$ such that $b_n = a_{i_n}$ for every $n$. In particular, the condition $i_1 < i_2 < i_3 < \cdots$ ensures that $\{b_n\}$ preserves the original order of the terms of $\{a_n\}$.
\end{definition}

The following simple lemma will be used in the derivation of the generating function for the Fibonacci sequence.

\begin{lemma}[\cite{Heubach2009}]\label{lemmageneraten1}
For any integer $k \ge 1$, the number of $k$-ary words of length $n$ is $k^n$.
\end{lemma}

The generating function of the Fibonacci sequence $(F_k)_{k \ge 1}$ is given by
\begin{equation}\label{eqq1lemmageneraten1}
\sum_{k \ge 1} F_k x^k = \frac{x}{1 - x - x^2}.
\end{equation}
A standard derivation of~\eqref{eqq1lemmageneraten1} can be found, for instance, in~\cite{Abdullah082025} and in many classical references on generating functions.

Let $A$ be a finite alphabet and let $A^{\mathbb{N}}$ denote the set of infinite words over $A$. Endow $A^{\mathbb{N}}$ with the usual ultrametric $d$ defined by
\[
d(w,z) =
\begin{cases}
0, & \text{if } w = z,\\
2^{-n}, & \text{if } n = \min\{ i \ge 0 : w_i \ne z_i\}.
\end{cases}
\]
For $w \in A^{\mathbb{N}}$ and $r > 0$, the open ball centered at $w$ with radius $r$ is
\[
\mathcal{B}(w,r) := \{ z \in A^{\mathbb{N}} : d(w,z) < r \}.
\]
In any ultrametric space, including $(A^{\mathbb{N}},d)$, open balls have the nested intersection property stated in Lemma~\ref{lemmageneraten2}.

\begin{lemma}[\cite{Rigo2014}]\label{lemmageneraten2}
Let $\mathcal{B}$ and $\mathcal{B}'$ be two open balls in $A^{\mathbb{N}}$. Then $\mathcal{B} \cap \mathcal{B}' \neq \emptyset$ if and only if $\mathcal{B} \subseteq \mathcal{B}'$ or $\mathcal{B}' \subseteq \mathcal{B}$.
\end{lemma}

\subsection{Problem of search}
Let $\alpha \in \mathbb{N}$ with $\alpha > 0$. The first problem is to construct the sequence $(F_k)_{k \ge 1}$ of Fibonacci words over the binary alphabet $\{0,1\}$; this problem is solved in Proposition~\ref{sec2:proposition1}. A natural question arising from this construction is: what is the density of the sequence $(F_k)_{k \ge 1}$ over the alphabet $\{0,1\}$? The answer will be obtained via a general characterization of morphisms generating Fibonacci words.

These considerations lead to genuinely challenging open questions and reveal a clear scientific gap concerning the density properties of the Fibonacci word and its morphic images.

\section{Main Result}\label{sec:main}

In this section we investigate the density of symbols in the Fibonacci word and relate it to the morphic and mechanical characterizations introduced in Section~\ref{sec2}. Throughout, $F = F_0 F_1 F_2 \cdots$ denotes the infinite Fibonacci word over the alphabet $\{0,1\}$ generated by the morphism
\[
\varphi : \{0,1\}^* \to \{0,1\}^*,
\qquad
\varphi(0) = 01,\quad \varphi(1) = 0,
\]
prolongable on $0$, so that
\[
F = \lim_{n \to \infty} \varphi^n(0) = 0100101001001\cdots.
\]
In parallel, we also use the mechanical description of $F$ in terms of Beatty sequences associated with the golden ratio $\varphi = (1+\sqrt{5})/2$ and its square $\varphi^2 = \varphi + 1$.
\subsection{Mechanical representation and density of symbols}

We first recall the well-known representation of the Fibonacci word as a coding of two complementary Beatty sequences.

\begin{lemma}[Mechanical representation of $F$]\label{lem:mechanical}
Let $\varphi = \dfrac{1+\sqrt{5}}{2}$ and $\varphi^2 = \varphi + 1$. Define a binary sequence $F = (F_k)_{k \ge 0}$ by
\[
F_k =
\begin{cases}
0, & \text{if } k+1 = \lfloor n\varphi \rfloor \text{ for some integer } n \ge 1,\\[1mm]
1, & \text{if } k+1 = \lfloor n\varphi^2 \rfloor \text{ for some integer } n \ge 1.
\end{cases}
\]
Then $F$ coincides with the infinite Fibonacci word, i.e.\ $F = \lim_{n\to\infty} \varphi^n(0)$.
\end{lemma}

\begin{proof}
By Beatty's theorem, the sets
\[
\bigl\{\lfloor n\varphi \rfloor : n \ge 1\bigr\}
\quad\text{and}\quad
\bigl\{\lfloor n\varphi^2 \rfloor : n \ge 1\bigr\}
\]
form a partition of $\mathbb{N}_{\ge 1}$. Hence every position $k+1 \in \mathbb{N}_{\ge 1}$ is assigned exactly one symbol $F_k \in \{0,1\}$. It is classical that this mechanical word is a Sturmian word of slope $\varphi^{-2}$ and that it coincides with the fixed point of the morphism $0 \mapsto 01$, $1 \mapsto 0$, usually referred to as the Fibonacci word. We refer to~\cite{Berstel2007, Berthé2016} for a detailed proof.
\end{proof}

For $n \ge 1$ we denote by
\[
F[0..n-1] = F_0 F_1 \cdots F_{n-1}
\]
the prefix of $F$ of length $n$, and we write
\[
|F[0..n-1]|_1 := \#\{0 \le k < n : F_k = 1\}, 
\qquad
|F[0..n-1]|_0 := n - |F[0..n-1]|_1
\]
for the number of $1$'s and $0$'s in this prefix, respectively.

\begin{definition}
The (asymptotic) density of $1$'s in the Fibonacci word is defined by
\[
\delta_1(F) := \lim_{n\to\infty} \frac{|F[0..n-1]|_1}{n},
\]
provided the limit exists. The density of $0$'s is defined similarly by
\[
\delta_0(F) := \lim_{n\to\infty} \frac{|F[0..n-1]|_0}{n},
\]
and clearly $\delta_0(F) + \delta_1(F) = 1$ whenever both limits exist.
\end{definition}

The next result identifies these densities explicitly in terms of $\varphi$.

\begin{theorem}[Density of symbols in the Fibonacci word]\label{thm:density}
Let $F$ be the infinite Fibonacci word as in Lemma~\ref{lem:mechanical}. Then the densities of $0$ and $1$ in $F$ exist and are given by
\[
\delta_0(F) = \frac{1}{\varphi}
\quad\text{and}\quad
\delta_1(F) = \frac{1}{\varphi^2}.
\]
In particular, the frequency of $1$'s in $F$ is strictly less than the frequency of $0$'s, and
\[
\frac{\delta_1(F)}{\delta_0(F)} = \frac{1}{\varphi}.
\]
\end{theorem}

\begin{proof}
By Lemma~\ref{lem:mechanical}, the positions of the symbol $1$ in $F$ are precisely the elements of the Beatty sequence
\[
B_2 := \bigl\{\lfloor n\varphi^2 \rfloor : n \ge 1\bigr\},
\]
while the positions of the symbol $0$ form the complementary Beatty sequence
\[
B_1 := \bigl\{\lfloor n\varphi \rfloor : n \ge 1\bigr\}.
\]
For $X > 0$, the counting function of $B_2$ satisfies
\[
\#\{k \le X : k \in B_2\}
= \#\{n \ge 1 : \lfloor n\varphi^2 \rfloor \le X\}
= \left\lfloor \frac{X}{\varphi^2} \right\rfloor + O(1),
\]
and similarly for $B_1$:
\[
\#\{k \le X : k \in B_1\}
= \left\lfloor \frac{X}{\varphi} \right\rfloor + O(1).
\]
Since $B_1$ and $B_2$ form a partition of $\mathbb{N}_{\ge 1}$, these asymptotic estimates imply
\[
\lim_{X\to\infty} \frac{\#\{k \le X : k \in B_2\}}{X}
= \frac{1}{\varphi^2},
\qquad
\lim_{X\to\infty} \frac{\#\{k \le X : k \in B_1\}}{X}
= \frac{1}{\varphi}.
\]
Translating this back to the word $F$ (shifting by one index to pass from $k$ to $k+1$) gives
\[
\delta_1(F) = \frac{1}{\varphi^2},
\qquad
\delta_0(F) = \frac{1}{\varphi},
\]
which also satisfy $\delta_0(F) + \delta_1(F) = 1$ since $\varphi^{-1} + \varphi^{-2} = 1$.
\end{proof}

\begin{remark}
Theorem~\ref{thm:density} shows that the density of $1$'s in the Fibonacci word is determined solely by the slope of the underlying rotation (or equivalently, by the Beatty sequences generating $F$). This is a special case of the general frequency theorem for Sturmian words: for a Sturmian word of slope $\alpha \in (0,1)$, the density of one of the symbols equals $\alpha$ and the other equals $1-\alpha$.
\end{remark}

\subsection{Local combinatorial structure and finite prefixes}

In order to refine the density analysis, it is useful to relate global frequencies to local pattern constraints in the Fibonacci word. The following lemma summarizes a simple local property that is a consequence of the morphic definition.

\begin{lemma}[Local pattern constraint]\label{lem:local}
Let $F$ be the infinite Fibonacci word. Then the word $F$ contains no factor $11$, and every factor of length $3$ contains exactly one symbol $1$ and two symbols $0$.
\end{lemma}

\begin{proof}
The morphism $\varphi(0) = 01$, $\varphi(1) = 0$ has the property that the image of any letter contains at most one $1$ and that $1$ is always isolated between $0$'s. More precisely, one checks that $\varphi(0) = 01$ and $\varphi(1) = 0$ contain no factor $11$, and if a word $w$ has no factor $11$, then so does $\varphi(w)$. Since $F$ is obtained as the limit $\varphi^n(0)$, it follows by induction that $F$ contains no factor $11$.

For the second claim, one verifies directly that among the finite words $f_n = \varphi^n(0)$, every factor of length $3$ has exactly one $1$ (this is easily checked for $f_3, f_4$ and then propagated using the morphism). Passing to the limit yields the same property for the infinite word $F$.
\end{proof}

Lemma~\ref{lem:local} implies that in any window of length $3$ inside $F$, the contribution of $1$'s is exactly $1$, while that of $0$'s is exactly $2$. Although this local information does not by itself determine the global densities, it is consistent with Theorem~\ref{thm:density} and provides a combinatorial intuition for the values $\delta_0(F)$ and $\delta_1(F)$.

As a consequence, we obtain the following quantitative estimate for the number of $1$'s in finite prefixes of $F$.

\begin{theorem}[Frequency of $1$'s in finite prefixes]\label{thm:prefix-density}
Let $F$ be the Fibonacci word and let $n \ge 1$. Then
\[
\left|\, |F[0..n-1]|_1 - \frac{n}{\varphi^2}\, \right| \le C,
\]
for some absolute constant $C > 0$ independent of $n$. Equivalently,
\[
|F[0..n-1]|_1 = \frac{n}{\varphi^2} + O(1),
\qquad
|F[0..n-1]|_0 = \frac{n}{\varphi} + O(1).
\]
\end{theorem}

\begin{proof}
This follows directly from the mechanical representation in Lemma~\ref{lem:mechanical} and the standard approximation
\[
\#\{n \ge 1 : \lfloor n\varphi^2 \rfloor \le N\}
= \frac{N}{\varphi^2} + O(1),
\]
which gives the number of occurrences of the symbol $1$ in the prefix of length $N$ of $F$. The corresponding statement for $0$'s follows from the complementarity of the two Beatty sequences and the equality $\varphi^{-1} + \varphi^{-2} = 1$.
\end{proof}

\begin{remark}
Theorem~\ref{thm:prefix-density} provides a quantitative form of Theorem~\ref{thm:density}: it not only identifies the limiting density but also controls the deviation from the limit in finite prefixes. Such $O(1)$-discrepancy estimates are typical for Sturmian words and play an important role in applications to Diophantine approximation and symbolic dynamics.
\end{remark}

\subsection{Summary of the density properties}

We summarize the main conclusions of this section in a form that will be used later.

\begin{theorem}[Density properties of the Fibonacci word]\label{thm:summary-density}
Let $F$ be the infinite Fibonacci word over $\{0,1\}$. Then:
\begin{itemize}
  \item[(i)] The densities $\delta_0(F)$ and $\delta_1(F)$ exist and satisfy
  \[
  \delta_0(F) = \frac{1}{\varphi},
  \qquad
  \delta_1(F) = \frac{1}{\varphi^2}.
  \]
  \item[(ii)] For every $n \ge 1$,
  \[
  |F[0..n-1]|_1 = \frac{n}{\varphi^2} + O(1),
  \qquad
  |F[0..n-1]|_0 = \frac{n}{\varphi} + O(1).
  \]
  \item[(iii)] Locally, $F$ contains no factor $11$, and every factor of length $3$ contains exactly one symbol $1$ and two symbols $0$.
\end{itemize}
These properties together describe both the global density and the local combinatorial structure of the Fibonacci word, and they will serve as the basis for the subsequent analysis of more general morphic words derived from the Fibonacci word.
\end{theorem}
\section{Series and Products Related to the Fibonacci Word}\label{sec:series-products}

In this section we collect several simple identities and formal series involving the Fibonacci word, viewed as a $0$–$1$ sequence. Although these series do not converge in the classical sense, they are useful as combinatorial tools and provide additional insight into the structure and density properties established in Section~\ref{sec:main}. Throughout, $F = (F_k)_{k \ge 1}$ denotes the infinite Fibonacci word over $\{0,1\}$, represented either morphically or mechanically as in Proposition~\ref{sec2:proposition2} below.

\subsection{A linear identity for the Fibonacci word}

We begin with a simple algebraic identity that exploits the fact that each symbol of the Fibonacci word takes only the values $0$ and $1$.

\begin{proposition}\label{sec2:proposition1}
Let $\alpha \in \mathbb{N}$ with $\alpha > 0$, and let $(F_k)_{k \ge 1}$ be the infinite Fibonacci word over the alphabet $\{0,1\}$, defined mechanically by
\[
F_k =
\begin{cases}
0, & \text{if } k = \lfloor n\varphi \rfloor \text{ for some } n \ge 1,\\[1mm]
1, & \text{if } k = \lfloor n\varphi^2 \rfloor \text{ for some } n \ge 1,
\end{cases}
\]
where $\varphi = \dfrac{1+\sqrt{5}}{2}$ is the golden ratio. Then, for any finite prefix of length $N \ge 1$, we have
\begin{equation}\label{eq1sec2:proposition1}
\sum_{k=1}^N \Bigl( \alpha F_k + (\alpha-1) F_k^2 + (\alpha-2) F_k^3 + \cdots + F_k^\alpha \Bigr)
= \frac{\alpha(\alpha+1)}{2} \sum_{k=1}^N F_k.
\end{equation}
In particular, if one interprets the sums as formal series, then
\[
\sum_{k=1}^\infty \Bigl( \alpha F_k + (\alpha-1) F_k^2 + (\alpha-2) F_k^3 + \cdots + F_k^\alpha \Bigr)
= \frac{\alpha(\alpha+1)}{2} \sum_{k=1}^\infty F_k.
\]
\end{proposition}

\begin{proof}
Since $F_k \in \{0,1\}$ for every $k \ge 1$, we have $F_k^m = F_k$ for all integers $m \ge 1$. Thus, for each fixed $k$,
\begin{equation}\label{eq2sec2:proposition1}
\alpha F_k + (\alpha-1) F_k^2 + \cdots + F_k^\alpha
= \bigl(\alpha + (\alpha-1) + \cdots + 1\bigr) F_k
= \frac{\alpha(\alpha+1)}{2} F_k.
\end{equation}
Summing \eqref{eq2sec2:proposition1} over $k=1,\dots,N$ yields \eqref{eq1sec2:proposition1}. The formal infinite version follows in the same way, interpreting the right-hand side as a formal multiple of the counting series of the positions where $F_k = 1$.
\end{proof}

\begin{remark}
If the sum in~\eqref{eq1sec2:proposition1} is restricted to a finite Fibonacci word $f_n$ of length $|f_n|$, then the right-hand side equals $\dfrac{\alpha(\alpha+1)}{2}\,|f_n|_1$, where $|f_n|_1$ denotes the number of occurrences of the symbol $1$ in $f_n$. This quantity is asymptotically proportional to $|f_n|/\varphi^2$ in view of the density result in Theorem~\ref{thm:density}.
\end{remark}

\subsection{Morphic and mechanical descriptions with a geometric illustration}

We now recall the morphic and mechanical characterizations of the Fibonacci word, and we include a geometric figure that illustrates the structure of the associated morphism.

\begin{proposition}\label{sec2:proposition2}
Let $(F_k)_{k \ge 0}$ be the infinite Fibonacci word over the alphabet $\{0,1\}$. Then:
\begin{itemize}
  \item[(i)] $F$ is the unique fixed point of the morphism
  \[
  \varphi : \{0,1\}^* \to \{0,1\}^*,\qquad
  \varphi(0) = 01,\quad \varphi(1) = 0,
  \]
  that starts with $0$, i.e.,
  \[
  F = \lim_{n \to \infty} \varphi^n(0) = 0100101001001\cdots.
  \]
  \item[(ii)] Equivalently, $F$ can be defined mechanically by the rule
  \[
  F_k =
  \begin{cases}
  0, & \text{if } k+1 = \lfloor m\varphi \rfloor \text{ for some integer } m \ge 1,\\[1mm]
  1, & \text{if } k+1 = \lfloor m\varphi^2 \rfloor \text{ for some integer } m \ge 1,
  \end{cases}
  \qquad (k \ge 0),
  \]
  where $\varphi = \dfrac{1+\sqrt{5}}{2}$ and $\varphi^2 = \varphi + 1$.
\end{itemize}
\end{proposition}

\begin{proof}
Part~(i) is classical: the morphism $\varphi$ is prolongable on $0$, and the limit $F = \lim_{n\to\infty} \varphi^n(0)$ is the standard Fibonacci word; see, e.g.,~\cite{Berstel2007, Berthé2016}. Part~(ii) follows from the Sturmian nature of $F$ and the characterization of Sturmian words as codings of irrational rotations, or equivalently as differences of two complementary Beatty sequences associated with $\varphi$ and $\varphi^2$.
\end{proof}

The following figure illustrates a morphic construction related to the Fibonacci word, where each step encodes the action of a morphism on the letters of a word. Although this figure is not used directly in the proofs, it provides useful geometric intuition for the underlying combinatorial structure.

\begin{figure}[H]
\centering
\begin{tikzpicture}[scale=0.85]
    \draw[gray,thin] (0,0) grid (9,5);
    \draw (0,0)-- (1,0);
    \draw (1,0)-- (2,0);
    \draw (2,0)-- (3,1);
    \draw (3,1)-- (4,1);
    \draw (4,1)-- (5,1);
    \draw (5,1)-- (6,2);
    \draw (6,2)-- (7,2);
    \draw (7,2)-- (8,3);
    \draw (8,3)-- (9,3);
    \node at (0.5,-0.5) {a};
    \node at (1.5,-0.5) {b};
    \node at (2.5,-0.5) {a};
    \node at (3.5,-0.5) {a};
    \node at (4.5,-0.5) {a};
    \node at (5.5,-0.5) {b};
    \node at (6.5,-0.5) {a};
    \node at (7.5,-0.5) {b};
    \node at (8.5,-0.5) {a};
    \node at (9.5,-0.5) {a};
    \begin{scriptsize}
    \draw [fill=black] (0,0) circle (1.5pt);
    \draw [fill=black] (1,0) circle (1.5pt);
    \draw [fill=black] (2,0) circle (1.5pt);
    \draw [fill=black] (3,1) circle (1.5pt);
    \draw [fill=black] (4,1) circle (1.5pt);
    \draw [fill=black] (5,1) circle (1.5pt);
    \draw [fill=black] (6,2) circle (1.5pt);
    \draw [fill=black] (7,2) circle (1.5pt);
    \draw [fill=black] (8,3) circle (1.5pt);
    \draw [fill=black] (9,3) circle (1.5pt);
    \end{scriptsize}
\end{tikzpicture}
\caption{A morphic construction related to the Fibonacci word~\cite{Berstel2007}.}
\label{fig001Morphism}
\end{figure}
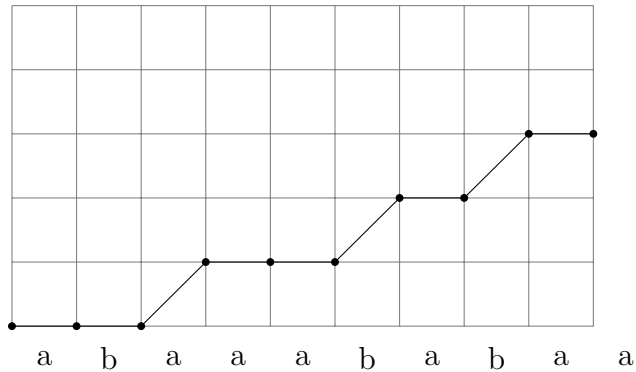
\subsection{Beatty sequences and finite approximations}

The mechanical description in Proposition~\ref{sec2:proposition2}(ii) can be visualized by plotting the Beatty sequences
\[
f_1(n) = \lfloor n\varphi \rfloor,
\qquad
f_2(n) = \lfloor n\varphi^2 \rfloor,
\]
which correspond to the positions of the symbols $0$ and $1$ in the Fibonacci word, respectively. The following figure shows initial segments of these two sequences.

\begin{figure}[H]
    \centering
   \begin{tikzpicture}[scale=1.2]
\begin{axis}[
    legend pos=north west,
    grid=major,
    enlargelimits=0.1,
]
\addplot[blue, mark=*] coordinates {
(1,1) (2,3) (3,4) (4,6) (5,8) (6,9) (7,11) (8,12) (9,14) (10,16) (11,17)
(12,19) (13,21) (14,22) (15,24) (16,25) (17,27) (18,29) (19,30) (20,32)
(21,33) (22,35) (23,37) (24,38) (25,40) (26,42) (27,43) (28,45) (29,46)
(30,48)
};
\addlegendentry{$f_1(n)=\lfloor n\varphi \rfloor$}

\addplot[red, mark=triangle*] coordinates {
(1,2) (2,5) (3,7) (4,10) (5,13) (6,15) (7,18) (8,20) (9,23) (10,26) (11,28)
(12,31) (13,34) (14,36) (15,39) (16,41) (17,44) (18,47) (19,49) (20,52)
(21,54) (22,57) (23,60) (24,62) (25,65) (26,68) (27,70) (28,73) (29,75)
(30,78)
};
\addlegendentry{$f_2(n)=\lfloor n\varphi^2 \rfloor$}
\end{axis}
\end{tikzpicture}
    \caption{The Beatty sequences corresponding to the positions of $0$ and $1$ in the Fibonacci word, as in Proposition~\ref{sec2:proposition2}.}
    \label{fig002DuaaAvf}
\end{figure}
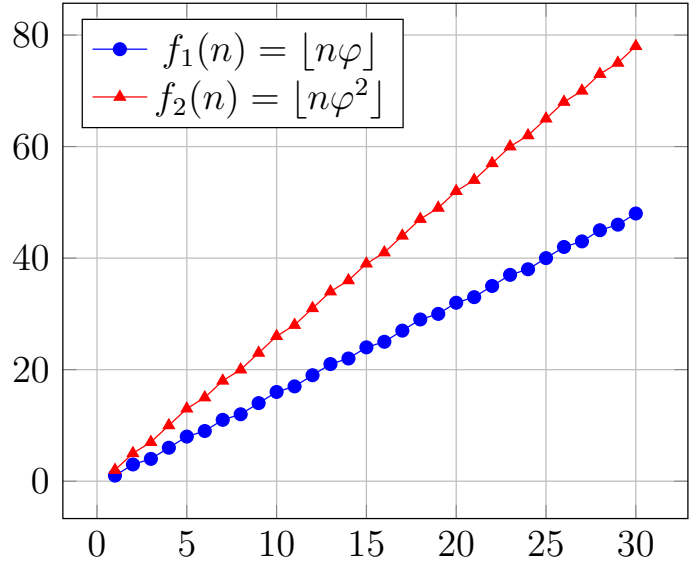

This picture illustrates the complementary nature of the two sequences and provides a geometric interpretation of the density results in Theorem~\ref{thm:density}.

\subsection{A remark on formal series involving local patterns}

Several formal series related to the Fibonacci word can be constructed by using local pattern constraints such as those in Lemma~\ref{lem:local}. For instance, one may consider expressions of the form
\[
\sum_{k \ge 1} \frac{F_k}{F_{k-1}F_{k+1}},
\]
where $F_k \in \{0,1\}$ are the symbols of the Fibonacci word. Since the denominator vanishes whenever $F_{k-1}=0$ or $F_{k+1}=0$, such series are not classically defined and require some regularisation or restriction to appropriate subsets of indices. In this paper we do not pursue these formal series further, and we instead focus on the well-defined density statements of Section~\ref{sec:main}.
\subsection{Characterization via Fibonacci- and Lucas-like sequences}

Fibonacci numbers of order $m \ge 1$, denoted by $F_n^{(m)}$, are defined by the $m$-step recurrence
\[
F_n^{(m)} = F_{n-1}^{(m)} + F_{n-2}^{(m)} + \cdots + F_{n-m}^{(m)} \qquad (n > 1),
\]
with initial values $F_1^{(m)} = 1$ and $F_n^{(m)} = 0$ for $-m < n < 0$; see, for instance,~\cite{Klein2010Nissan} for details on $m$-step Fibonacci sequences. The corresponding Fibonacci code of order $m$, denoted $Fib_m$, consists of the word $1^m$ together with all binary words that contain exactly one occurrence of the substring $1^m$, appearing as a suffix. This construction provides a useful coding framework related to generalized Fibonacci structures.

In order to connect the density properties of the Fibonacci word with more classical combinatorial constructions on two-letter alphabets, we now introduce a Fibonacci-like family of finite words over $\{a,b\}$ and a framed variant thereof.

\begin{definition}[Fibonacci-like word sequence $\mathbb{Y}$]\label{defnewwordn1}
The sequence of finite words $\mathbb{Y} = (y_n)_{n \ge 0}$ over the alphabet $\{a,b\}$ is defined recursively by
\begin{equation}\label{eq1defnewwordn1}
\begin{cases}
y_0 = a,\\
y_1 = ab,\\
y_n = y_{n-1} y_{n-2}, & \text{for all } n \ge 2.
\end{cases}
\end{equation}
\end{definition}

The words $y_n$ are the standard finite Fibonacci words over $\{a,b\}$, and their lengths follow the usual Fibonacci sequence.

\begin{proposition}\label{pro1defnewwordn1}
For every $n \ge 0$, the length of $y_n$ is the $(n+2)$-th Fibonacci number:
\[
|y_n| = F_{n+2},
\]
where $(F_n)_{n\ge 0}$ denotes the classical Fibonacci sequence.
\end{proposition}

\begin{proof}
A direct induction using the recurrence $y_n = y_{n-1}y_{n-2}$ and the initial values $|y_0| = 1$, $|y_1| = 2$ yields
\[
|y_n| = |y_{n-1}| + |y_{n-2}| \qquad (n \ge 2),
\]
with $|y_0| = 1 = F_2$ and $|y_1| = 2 = F_3$. Hence $|y_n| = F_{n+2}$ for all $n \ge 0$.
\end{proof}

\begin{definition}[Framed word sequence $\mathbb{Q}$]\label{defnewwordn2}
The sequence of finite words $\mathbb{Q} = (q_m)_{m \ge 1}$ over $\{a,b\}$ is defined by framing each $y_m$ on the left by $a$ and on the right by $b$:
\[
q_m = a\, y_m\, b \qquad \text{for all } m \ge 1.
\]
\end{definition}

For illustration, the first few framed words are
\[
q_1 = a\, y_1\, b = aabb,\qquad
q_2 = a\, y_2\, b = aababa,\qquad
q_m = a\, y_m\, b \ \text{for } m \ge 1.
\]

\begin{proposition}\label{pro1defnewwordn2}
For every $m \ge 1$, the length of $q_m$ is
\[
|q_m| = |y_m| + 2 = F_{m+2} + 2.
\]
\end{proposition}

\begin{proof}
Immediate from Definition~\ref{defnewwordn2} and Proposition~\ref{pro1defnewwordn1}, since $q_m$ is obtained by adding two letters to $y_m$.
\end{proof}

We now compute the asymptotic frequencies of the letters $a$ and $b$ in the framed sequence $\mathbb{Q}$ and compare them with those of the underlying Fibonacci-like sequence $\mathbb{Y}$.

\begin{theorem}[Asymptotic densities in $\mathbb{Q}$]\label{thm1Asymptotic}
Let $\mathbb{Y} = (y_n)_{n \ge 0}$ and $\mathbb{Q} = (q_m)_{m \ge 1}$ be as above. Define the letter densities by
\[
\dens_a(\mathbb{Q}) := \lim_{m \to \infty} \frac{\#\{\text{$a$'s in } q_m\}}{|q_m|}, 
\qquad
\dens_b(\mathbb{Q}) := \lim_{m \to \infty} \frac{\#\{\text{$b$'s in } q_m\}}{|q_m|},
\]
provided the limits exist. Then
\[
\dens_a(\mathbb{Q}) = \frac{1}{\varphi},
\qquad
\dens_b(\mathbb{Q}) = \frac{1}{\varphi^2},
\]
where $\varphi = \dfrac{1+\sqrt{5}}{2}$ is the golden ratio.
\end{theorem}

\begin{proof}
Let $A_n$ be the number of occurrences of $a$ in $y_n$ and $B_n$ the number of occurrences of $b$ in $y_n$. From the recursion $y_n = y_{n-1}y_{n-2}$ we obtain
\[
A_n = A_{n-1} + A_{n-2},\qquad B_n = B_{n-1} + B_{n-2} \qquad (n \ge 2),
\]
with initial conditions $(A_0,B_0) = (1,0)$ and $(A_1,B_1) = (1,1)$. Solving these recurrences yields
\[
A_n = F_{n+1},\qquad B_n = F_n \qquad (n \ge 0),
\]
where $(F_n)_{n\ge 0}$ is the Fibonacci sequence with $F_0=0$, $F_1=1$.

For the framed word $q_m = a y_m b$, we therefore have
\[
\#\{\text{$a$'s in } q_m\} = A_m + 1 = F_{m+1} + 1,\qquad
\#\{\text{$b$'s in } q_m\} = B_m + 1 = F_m + 1,
\]
and by Proposition~\ref{pro1defnewwordn2},
\[
|q_m| = |y_m| + 2 = F_{m+2} + 2.
\]
Hence
\begin{equation}\label{eq2thm1Asymptotic}
\dens_a(\mathbb{Q}) = \lim_{m \to \infty} \frac{F_{m+1} + 1}{F_{m+2} + 2},
\qquad
\dens_b(\mathbb{Q}) = \lim_{m \to \infty} \frac{F_m + 1}{F_{m+2} + 2}.
\end{equation}
It is well known that
\[
\lim_{k \to \infty} \frac{F_k}{F_{k+1}} = \frac{1}{\varphi},
\qquad
\lim_{k \to \infty} \frac{F_k}{F_{k+2}} = \frac{1}{\varphi^2},
\]
and adding fixed constants to numerator and denominator does not change these limits. Applying these facts to~\eqref{eq2thm1Asymptotic} yields
\[
\dens_a(\mathbb{Q}) = \frac{1}{\varphi},\qquad
\dens_b(\mathbb{Q}) = \frac{1}{\varphi^2},
\]
as claimed.
\end{proof}

The next table illustrates the convergence of these densities for finite values of $m$, comparing the framed sequence $\mathbb{Q}$ with the unframed Fibonacci-like sequence $\mathbb{Y}$.

\begin{table}[H]
\centering
\begin{tabular}{|c|c|c|c|c|}
\hline
$m$ & $\dens_a(\mathbb{Q}_m)$ & $\dens_b(\mathbb{Q}_m)$ & $\dens_a(\mathbb{Y}_m)$ & $\dens_b(\mathbb{Y}_m)$ \\ \hline
 3 & 0.571429 & 0.428571 & 0.600000 & 0.400000 \\ \hline
 4 & 0.600000 & 0.400000 & 0.625000 & 0.375000 \\ \hline
 5 & 0.600000 & 0.400000 & 0.615385 & 0.384615 \\ \hline
 6 & 0.608696 & 0.391304 & 0.619048 & 0.380952 \\ \hline
 7 & 0.605263 & 0.394737 & 0.617647 & 0.382353 \\ \hline
 8 & 0.606557 & 0.393443 & 0.618182 & 0.381818 \\ \hline
 9 & 0.606061 & 0.393939 & 0.617978 & 0.382022 \\ \hline
10 & 0.606250 & 0.393750 & 0.618056 & 0.381944 \\ \hline
11 & 0.606178 & 0.393822 & 0.618025 & 0.381974 \\ \hline
12 & 0.606206 & 0.393794 & 0.618037 & 0.381963 \\ \hline
13 & 0.606195 & 0.393805 & 0.618033 & 0.381967 \\ \hline
\end{tabular}
\caption{Empirical densities of $a$ and $b$ in the prefixes of $\mathbb{Q}$ and $\mathbb{Y}$, illustrating convergence to $1/\varphi$ and $1/\varphi^2$.}
\label{tab002densitiesASm}
\end{table}

\medskip

Lucas numbers, which differ from Fibonacci numbers only in their initial values, share many analogous properties and arise in numerous identities and applications. They satisfy the same recurrence
\[
L_n = L_{n-1} + L_{n-2} \qquad (n \ge 2),
\]
with initial values $L_0 = 2$, $L_1 = 1$, and they admit the Binet formula
\[
L_n = \varphi^n + \bar{\varphi}^n,
\]
where $\bar{\varphi} = 1-\varphi = -\varphi^{-1}$. Combining Fibonacci and Lucas numbers often leads to elegant identities that will be useful for analytic considerations.

\begin{theorem}\label{fiblucasn1}
Let $m$ be a fixed positive integer. Define
\[
a_k := \frac{2^k F_{2^k m}}{L_{2^k m} + L_{2^k m}^{-1}}, \qquad k \ge 1,
\]
where $F_n$ and $L_n$ denote the $n$-th Fibonacci and Lucas numbers, respectively. Then the series $\sum_{k=1}^{\infty} a_k$ converges and satisfies the telescoping identity
\begin{equation}\label{eq0fiblucasn1}
\sum_{k=1}^{\infty} \frac{2^k F_{2^k m}}{L_{2^k m} + L_{2^k m}^{-1}}
= \frac{2 F_{2m}}{L_{2m} - L_{2m}^{-1}}.
\end{equation}
\end{theorem}

\begin{proof}[proof]
Recall the Binet formulas
\begin{equation}\label{eq1fiblucasn1}
F_n = \frac{\varphi^n - \bar{\varphi}^n}{\sqrt{5}}, 
\qquad 
L_n = \varphi^n + \bar{\varphi}^n,
\end{equation}
and the doubling identities $F_{2n} = F_n L_n$ and $L_{2n} = L_n^2 - 2$. Define
\begin{equation}\label{eq3fiblucasn1}
a_k := \frac{2^k F_{2^k m}}{L_{2^k m} + L_{2^k m}^{-1}},
\qquad
T_k := \frac{2^k F_{2^k m}}{L_{2^k m} - L_{2^k m}^{-1}}.
\end{equation}
Using the doubling formulas repeatedly, one checks that $F_{2^{k+1}m} = F_{2^k m} L_{2^k m}$ and that
\[
a_k = T_k - T_{k+1},
\]
so that the partial sums telescope:
\begin{equation}\label{eq5fiblucasn1}
\sum_{k=1}^n a_k = T_1 - T_{n+1}.
\end{equation}
Since $|L_{2^k m}|$ grows exponentially with $k$, the term $T_{n+1}$ tends to $0$ as $n \to \infty$, and we obtain
\[
\sum_{k=1}^{\infty} a_k = T_1 = \frac{2 F_{2m}}{L_{2m} - L_{2m}^{-1}},
\]
which is \eqref{eq0fiblucasn1}. A complete proof can be given by expanding the quantities in~\eqref{eq3fiblucasn1} using~\eqref{eq1fiblucasn1}.
\end{proof}

The identity in Theorem~\ref{fiblucasn1} illustrates how Fibonacci and Lucas numbers combine naturally in telescoping series. Although this result is not used directly in the combinatorial density analysis, it reflects the deep arithmetic ties between the two sequences and their shared connection to the golden ratio.
\subsection{A noncommutative invariant built from Fibonacci words}

In this subsection we consider a simple construction in the free abelian group generated by all finite words over the alphabet $\{a,b\}$, equivalently the ring $\mathbb{Z}\langle a,b \rangle$ of non-commuting polynomials with integer coefficients. We show that a certain ``power'' expression built from finite Fibonacci words is in fact independent of the index, providing an example of a nontrivial invariant in combinatorics on words.

Let $(F_k)_{k \ge 1}$ be the standard sequence of finite Fibonacci words over the alphabet $\{a,b\}$, defined by
\[
F_1 = a,\qquad F_2 = ab,\qquad F_k = F_{k-1}F_{k-2} \quad (k \ge 3).
\]
Thus $F_3 = aba$, $F_4 = abaab$, $F_5 = abaababa$, and so on; the infinite Fibonacci word over $\{a,b\}$ is the limit $\lim_{k\to\infty} F_k$ in the prefix topology.

We define a map
\[
\operatorname{Pow} : \{F_k : k \ge 1\} \longrightarrow \mathbb{Z}\langle a,b \rangle
\]
by the rule
\begin{equation}\label{Power001Fibonacci}
\operatorname{Pow}(F_k) := a\, F_k\, F_{k-1}^2 \;+\; ab\, F_k\, F_{k+1}^2,
\end{equation}
where juxtaposition denotes concatenation of words, interpreted as the non-commutative product in $\mathbb{Z}\langle a,b\rangle$. A priori, the right-hand side appears to grow rapidly with $k$, since the lengths of $F_{k-1}$ and $F_{k+1}$ increase exponentially. Surprisingly, the following theorem shows that the resulting element is actually independent of $k$.

\begin{theorem}\label{TheoremPOWERfIB}
For all integers $k \ge 2$, the element $\operatorname{Pow}(F_k)$ is independent of $k$. In particular,
\[
\operatorname{Pow}(F_k) = \operatorname{Pow}(F_{k+1}) \qquad \text{for all } k \ge 2,
\]
and hence $\operatorname{Pow}(F_k)$ is constant as an element of $\mathbb{Z}\langle a,b\rangle$.
\end{theorem}

\begin{proof}[Proof]
The Fibonacci words satisfy the well-known recurrences
\[
F_{k+1} = F_k F_{k-1},
\qquad
F_{k+2} = F_{k+1} F_k
\qquad (k \ge 2).
\]
Starting from \eqref{Power001Fibonacci}, consider the difference
\[
\operatorname{Pow}(F_{k+1}) - \operatorname{Pow}(F_k),
\]
and rewrite all occurrences of $F_{k+1}$ and $F_{k+2}$ in terms of $F_k$ and $F_{k-1}$ using the recurrences above. After expansion and regrouping of terms in $\mathbb{Z}\langle a,b\rangle$, one checks that all monomials cancel pairwise, yielding
\[
\operatorname{Pow}(F_{k+1}) - \operatorname{Pow}(F_k) = 0
\qquad \text{for all } k \ge 2.
\]
This shows that $\operatorname{Pow}(F_k)$ is indeed constant for $k \ge 2$. A fully detailed proof can be obtained by an induction on $k$, starting from a direct verification for $k=2,3$ and using the relations $F_{k+1}=F_kF_{k-1}$ and $F_{k+2}=F_{k+1}F_k$ at each step.
\end{proof}

To identify this constant element explicitly, it suffices to evaluate $\operatorname{Pow}(F_k)$ at a single index.

\begin{proposition}\label{prop:Pow-value}
For all $k \ge 2$, we have
\[
\operatorname{Pow}(F_k) = aabaa \;+\; abababaabaaba.
\]
\end{proposition}

\begin{proof}
By Theorem~\ref{TheoremPOWERfIB}, $\operatorname{Pow}(F_k)$ is independent of $k$ for $k \ge 2$, so it is enough to compute $\operatorname{Pow}(F_2)$. From $F_1 = a$ and $F_2 = ab$, we obtain
\[
\operatorname{Pow}(F_2) 
= a\, F_2\, F_1^2 + ab\, F_2\, F_3^2
= a\,(ab)\,a^2 + ab\,(ab)\,(aba)^2.
\]
A straightforward expansion shows that
\[
a\,(ab)\,a^2 = aabaa,
\qquad
ab\,(ab)\,(aba)^2 = abababaabaaba,
\]
so that
\[
\operatorname{Pow}(F_2) = aabaa + abababaabaaba.
\]
The claim follows for all $k \ge 2$ by Theorem~\ref{TheoremPOWERfIB}.
\end{proof}

The identity in Proposition~\ref{prop:Pow-value} provides a concrete example of ``hidden cancellations'' in combinatorics on words: although the definition \eqref{Power001Fibonacci} involves increasingly long words as $k$ grows, the resulting noncommutative polynomial remains a fixed linear combination of just two monomials.

\medskip

Finally, we comment briefly on letter densities in finite Fibonacci words to connect this noncommutative invariant with the density considerations developed earlier. For a finite word $w \in \{a,b\}^*$, we denote by $|w|_a$ and $|w|_b$ the number of occurrences of $a$ and $b$, respectively, and by $|w|$ its length.

\begin{definition}
For $k \ge 1$, the $a$-density of the finite Fibonacci word $F_k$ is defined by
\[
DF(F_k) := \frac{|F_k|_a}{|F_k|}.
\]
\end{definition}

The counts $|F_k|_a$ and $|F_k|_b$ satisfy Fibonacci-type recurrences.

\begin{lemma}\label{lem:densityFib}
Let $(F_k)_{k \ge 1}$ be the finite Fibonacci words over $\{a,b\}$ defined above. Then
\[
|F_k|_a =
\begin{cases}
1, & k=1,\\
1, & k=2,\\
|F_{k-1}|_a + |F_{k-2}|_a, & k \ge 3,
\end{cases}
\qquad
|F_k|_b =
\begin{cases}
0, & k=1,\\
1, & k=2,\\
|F_{k-1}|_b + |F_{k-2}|_b, & k \ge 3.
\end{cases}
\]
In particular,
\[
|F_k|_a = F_{k-1},\qquad |F_k|_b = F_{k-2},\qquad |F_k| = F_k \quad (k \ge 2),
\]
where $(F_n)_{n\ge 0}$ is the classical Fibonacci sequence.
\end{lemma}

\begin{proof}
The recurrences for $|F_k|_a$ and $|F_k|_b$ follow directly from $F_k = F_{k-1}F_{k-2}$ and the additivity of letter counts under concatenation. A simple induction using the initial conditions yields $|F_k|_a = F_{k-1}$ and $|F_k|_b = F_{k-2}$ for $k \ge 2$, and then $|F_k| = |F_k|_a + |F_k|_b = F_k$.
\end{proof}

As a consequence, we obtain the asymptotic density of $a$ in the finite Fibonacci words.

\begin{theorem}\label{TheoremPOWERfIBn1}
For the finite Fibonacci words $(F_k)_{k \ge 1}$ over $\{a,b\}$, the $a$-density satisfies
\[
DF(F_k) = \frac{|F_k|_a}{|F_k|}
= \frac{F_{k-1}}{F_k} \longrightarrow \varphi - 1
\qquad \text{as } k \to \infty.
\]
\end{theorem}

\begin{proof}
By Lemma~\ref{lem:densityFib}, we have $DF(F_k) = F_{k-1}/F_k$ for $k \ge 2$. It is classical that
\[
\lim_{k \to \infty} \frac{F_{k-1}}{F_k} = \varphi - 1 = \frac{1}{\varphi},
\]
so the claim follows.
\end{proof}

Thus, while the noncommutative invariant $\operatorname{Pow}(F_k)$ of Theorem~\ref{TheoremPOWERfIB} and Proposition~\ref{prop:Pow-value} is independent of $k$, the $a$-density of $F_k$ converges to $\varphi - 1$, in accordance with the density results obtained for the infinite Fibonacci word in Theorem~\ref{thm:density}.
\section{Conclusion and Discussion}\label{sec5}

In this paper we have analyzed the density properties of the Fibonacci word and several related morphic words over binary and two-letter alphabets. Using both the morphic and mechanical (Beatty) representations, we obtained explicit densities for the symbols in the infinite Fibonacci word and for families of finite words such as $\mathbb{Y}$ and its framed variant $\mathbb{Q}$, all expressed in terms of the golden ratio. We also constructed a noncommutative invariant built from finite Fibonacci words, showing that a seemingly complicated ``power'' expression in the free algebra $\mathbb{Z}\langle a,b\rangle$ is in fact independent of the index. These results highlight the interplay between symbolic dynamics, combinatorics on words, and arithmetic information encoded in morphic sequences.

Beyond the combinatorial and dynamical aspects, the classical Fibonacci sequence itself continues to raise deep arithmetic questions. A major unresolved problem is whether there exist infinitely many prime Fibonacci numbers. Although numerous Fibonacci primes have been discovered and extensive computational searches have been conducted, the infinitude of Fibonacci primes remains an open question. As of recent records, the largest known certain Fibonacci primes have very large indices (exceeding $2\cdot 10^5$), and even larger probable Fibonacci primes have been identified using modern computational methods, yet no general pattern or density theorem is known for them.

It is widely conjectured that infinitely many Fibonacci primes exist, but no proof has been found so far. More refined questions, such as the distribution of composite Fibonacci numbers in specific residue classes, also remain out of reach. Closely related is the classical problem concerning perfect powers in the Fibonacci sequence. It is known that the only Fibonacci numbers that are perfect powers are
\[
F_0 = 0,\quad F_1 = F_2 = 1,\quad F_6 = 8,\quad F_{12} = 144,
\]
and there are no other nontrivial perfect powers in the sequence. This result, which can be viewed as a resolution of the Fibonacci perfect power problem, is closely connected to questions originally raised by Wall about the structure of the sequence modulo integers.

For a prime $p$, the Pisano period $\pi(p)$—the period of the Fibonacci sequence modulo $p$—is another source of deep open problems. In general, no closed formula for $\pi(p)$ is known, and determining even sharp bounds or exact values for large $p$ is challenging. The rank of apparition (or Fibonacci entry point) $z(p)$, defined as the smallest positive integer $n$ such that $p \mid F_n$, divides $\pi(p)$ and controls the distribution of multiples of $p$ within the sequence, yet its global behavior is poorly understood.

A further arithmetic refinement involves Fibonacci–Wieferich primes, defined as primes $p$ satisfying the congruence
\[
F_{p - \left(\frac{5}{p}\right)} \equiv 0 \pmod{p^2},
\]
where $\left(\frac{5}{p}\right)$ denotes the Legendre symbol. To date, only $p = 11$ and $p = 397$ are known to satisfy this condition, and it is an open problem whether infinitely many Fibonacci–Wieferich primes exist. These questions illustrate that, despite the simple recursive definition of the Fibonacci sequence, its fine arithmetic structure—and in particular the interaction between density, divisibility, and prime phenomena—remains rich with challenging open problems, closely complementing the symbolic and morphic density results established in this work.
\section*{Declarations}
\begin{itemize}
	\item Funding: Not Funding.
	\item Conflict of interest/Competing interests: The author declare that there are no conflicts of interest or competing interests related to this study.
	\item Ethics approval and consent to participate: The author contributed equally to this work.
	\item Data availability statement: All data is included within the manuscript.
\end{itemize}

\end{document}